\documentclass[11pt]{amsart}
\advance\hoffset by -0.5 truecm
\usepackage{amssymb}
\usepackage{amsmath}
\usepackage{graphicx}
\usepackage{color}

\newtheorem{Theorem}{Theorem}[section]
\newtheorem{Lemma}[Theorem]{Lemma}
\newtheorem{Corollary}[Theorem]{Corollary}

\newtheorem{Proposition}[Theorem]{Proposition}

\newtheorem{Definition}[Theorem]{Definition}
\newtheorem{Example}[Theorem]{Example}
\newtheorem{Remark}[Theorem]{Remark}

\def\V{\mbox{Var}}

\def\R\re
\def\V{\bf V}

\def \re{{\mathbb R}}

\def \0{\lambda_{0}}

\begin{document}
\title[Isoparametric functions and nodal solutions]{Isoparametric functions and nodal solutions of the Yamabe equation }

\author[G. Henry]{G. Henry}\thanks{G. Henry is partially supported by grant PICT-2016-1302 from ANPCyT }
  \address{Departamento de Matem\'atica, FCEyN, Universidad de Buenos 
Aires, Ciudad Universitaria, Pab. I., C1428EHA,
           Buenos Aires, Argentina and CONICET Argentina.}
\email{ghenry@dm.uba.ar}

\subjclass{53C21}

\date{}


\begin{abstract} We prove existence results for nodal solutions of the Yamabe equation that are constant along the level sets of an isoparametric function.  
\end{abstract}

\maketitle

\section{Introduction}

Let $(M,g)$ be a closed Riemanian manifold of dimension $n\geq 3$.  We say that $u:M\longrightarrow\re$ is a solution of the Yamabe equation if $u$ satisfies  for some constant $c$ the equation
\begin{equation}\label{Yamabeeq}
a_n\Delta_gu+s_gu=c|u|^{p_n-2}u,
\end{equation}
 \noindent 
 where $a_n:=\frac{4(n-1)}{n-2}$, $p_n:=\frac{2n}{n-2}$ and $s_g$ is the scalar curvature of $(M,g)$.

The aim of this article is to study the existence of certain {\it{nodal  solutions}} ({\em{i.e.}}, sign changing solutions) of the Yamabe equation that arise by considering isoparametric functions.
 More precisely, we are interested in the following question. Let $f$ be an isoparametric function on a closed Riemannian manifold $(M,g)$ of dimension at least 3. Does there exist a nodal solution of the Yamabe equation that is constant along the level sets of $f$? The purpose of this  paper is to discuss  this problem if the manifold is either a Riemannian product, or each level set of  the isoparametric function has positive dimension.

Let $(M,g)$ be a closed connected Riemannian manifold. A smooth non-constant function $f:M\longrightarrow \re$ is called an {\it isoparametric function} is there exist a smooth function  $b$ and a continuous function $a$  such that
\begin{equation}\label{gradient}\| \nabla f \|^2 = b \circ f,
 \end{equation}
and
\begin{equation}\label{laplacian}
\Delta_{g} f = a\circ f.
\end{equation}

Let $f(M)=[t_{\min},t_{\max}]$. We denote by $M_t$ the level set of the function $f$ corresponding to the value $t$, that is $M_t=f^{-1}(t)$. The sets where $f$ attains the global maximum $M_{t_{\max}}$ and the global minimum $M_{t_{\min}}$  will be denoted by $M_{+}$ and  $M_{-}$, respectively.  The union of these sets is called the {\it focal set} of $f$. When $t$ is a regular value of $f$  we say that the level set  $M_t$ is an {\it isopara\-me\-tric hypersurface}.   The condition  \eqref{gradient} implies that the regular level sets of $f$  are parallel hypersurfaces. On the other hand, by  condition  \eqref{laplacian} under condition \eqref{gradient}, the level sets of $f$  have constant mean curvature. Wang  proved in \cite{Wang} that the only critical values of an isoparametric function on a closed Riemannian manifolds are its global maximum and minimum. In fact this is true for a {\it transnormal function}, {\em{i.e.}}, a non-constant function that satisfies the property \eqref{gradient}.  
Also it was proved in (Theorem A, \cite{Wang}) that the focal sets are actually submanifolds of $M$ (they might be non-connected  and with different dimensions in each connected component) and each re\-gular level set $M_t$ is a tube over either of the focal submanifolds.  Ge and Tang proved in \cite{Ge-Tang} that if the co-dimension of each components of the focal submanifolds of $f$ is at least 2 (which is  actually equivalent to the focal set of any regular level set of $f$ is exactly the focal submanifold $M_-\cup M_+$), then all the level sets are connected and there exists at least one isoparametric hypersurface induced by $f$ that is a minimal hypersurface. Moreover,  if the Ricci curvature is positive this minimal isoparametric hypersurface is unique, see (Corollary 2.1, \cite{Ge-Tang2}). The isoparametric functions with focal submanifolds of co-dimension greater than 1 are called {\it proper isoparametric functions}.  
It was proved in \cite{Ge-Tang2} that  the focal submanifolds  $M_-$ and $M_+$ of an proper isoparametric function are minimal submanifolds. Actually this result was proved by Miyaoka in \cite{Miyaoka} for  transnormal functions.

 The isoparametric functions and isopametric hypersurfaces have attained  lot of interest in the past. When the ambient manifold is a space form  the isoparametric hypersurfaces have constant principal curvatures as Cartan proved in \cite{Cartan}. The classification of isoparametric hypersurfaces in the Euclidean and in the  hyperbolic space was done in the first part of the twentieth-century by Levi-Civita \cite{Levi-Civita} and  Segre \cite{Segre}, and by  Cartan \cite{Cartan}, respectively. The affine hyperplanes,  the embedded spheres $S^{n-1}$ and $S^{n-k}\times \re^{k-1}$ are all the isoparametric hypersurfaces in the Euclidean setting. In the hyperbolic space an isoparametric hypersurface could be an hyperplane, $S^{n-k}\times H^{k-1}$, a geodesic sphere or a horoesphere. 
 But the theory  is much more complex in the case of the sphere  endowed with the standard round metric $(S^n,g^n_0)$. 
Notice that for the Euclidean and the hyperbolic space all the isoparametric hypersurfaces are homogeneous.  On the other hand in the round sphere, besides the homogeneous families there exist non-homogeoneous isoparametric hypersurfaces. The first examples of these kind of hypersurfaces were introduced by Ozeki and Takeuchi  in \cite{Ozeki-Takeuchi-I} and \cite{Ozeki-Takeuchi-II} (see also \cite{Ferus} for a generalization of this family of submanifolds). 
  The classification of isoparametric  hypersurfaces in the sphere was an outstanding   problem that remained open for a long time. Since the article by Cartan \cite{Cartan}  lots of works   have been carried out in order to achieve the total classification of isoparametric hypersurfaces in the sphere,  see for instance the articles  by  Munzner \cite{Munzner} and  \cite{Munzner2},  Abresch \cite{abresch}, Hsiang and Lawson \cite{Lawson}, Ferus et al. \cite{Ferus}, Stolz \cite{Stolz},  Cecil et al. \cite{Cecil}, Immervoll \cite{Immervoll}, Chi \cite{Chi1} and \cite{Chi2}, Miyaoka \cite{MiyaokaAnnals} and  \cite{MiyaokaAnnals2}. In a recent paper Chi \cite{Chi3}, classifies the isoparametric hypersurfaces in the sphere with four distinct principal curvatures and multiplicity pair $(7,8)$. This was the remaining case to obtain a total classification.
However, for general Riemannian manifolds the classification is far for being achieved. 

In the space forms there are many isoparametric functions and hypersurfaces.  Nevertheless, this is not the case for a general Riemannian manifold. It can be seen, by the structural theorem of isoparame\-tric hypersufaces stated above (Theorem A, \cite{Wang}), that there  might be obstructions to the existence of isoparametric functions.  Indeed, by a result of Miyaoka (Theorem 1.1, \cite{Miyaoka})  we know that a closed Riemannian manifold that admits an isoparametric function must be diffeomorphic to the union of two disc bundles over the focal submanifolds.

In \cite{Henry-Petean}, Petean and the author made use of the classification  of isoparametric functions on the sphere and bifurcation theory methods to prove multiplicity results of positive solutions of the Yamabe equation of Riemannian products with the sphere.  As we pointed out at the beginning of this section,  in this article we address the problem of whether an isoparametric function induces a nodal solution of the Yamabe equation.

 Positive solutions of the Yamabe equation are related with the exis\-tence of constant scalar curvature metrics in the conformal class $[g]$ of $g$, which is the set of Riemannian metrics of the form $\phi g$  with $\phi$  a po\-si\-tive smooth function on $M$. More precisely, if $u$ is a positive solution of Equation  \eqref{Yamabeeq},  then the Riemannian metric defined by  $h=u^{p_n-2}g$ has constant scalar curvature $c$. By the resolution of the celebrated  Yamabe problem (Yamabe \cite{Yamabe}, Tr\"udinger \cite{Trudinger}, Aubin \cite{Aubin}, and Schoen \cite{Schoen}), we know that there exists at least one positive solution of the Yamabe equation, and therefore, therein any conformal class there is a Riemannian metric of constant scalar curvature. Since the end of 1960s lots of mathematicians have been studying the space of positive solutions of the Yamabe equation and considerable progress has been made in  understanding  the  e\-xis\-tence and multiplicity of positive solutions of this equation.   For instance, it was shown by Khuri, Marques and  Schoen \cite{Khuri-Marques-Schoen} that the space of positive solutions of the Yamabe equation of a closed Riemannian manifold (unless it is the round sphere) is compact in the $C^2-$topology if the dimension of the manifold is at most 24. On the other hand,  Marques and Brendle \cite{Brendle-Marques1} proved that the compactness does not hold for dimension greater than 24. However, till recently quite little was known about nodal solutions of the Ya\-ma\-be equation.  Among the authors that have studied  nodal solutions of semi-linear elliptic equations on Riemannian manifolds we can mention Ding \cite{Ding}, Hebey and Vaugon \cite{Hebey-Vaugon2}, Holcman \cite{Holcman}, Jourdain \cite{Jourdain}, Djadli and Jourdain \cite{Djadli-Jourdain}, and del Pino et al. \cite{delPino}. But the first general result on the existence of nodal solution of Yamabe equation was obtained by Ammann and Humbert in \cite{Ammann-Humbert}. They introduced the so called second Yamabe constant which is a conformal invariant induced by the second eigenvalue of the conformal Laplacian operator  (see Section \ref{secondyamabe}, Preliminaries). Making use of this invariant  they were able to prove existence results when the conformal class of the Riemannian manifold admits a Riemannian metric of non-negative scalar curvature. Further results on the existence of nodal solutions of the Yamabe equation using the variational  approach of Ammann and Humbert can be found in Petean \cite{Petean2}, El Sayed \cite{ElSayed}, Madani and the author \cite{Henry-Madani}. We refer the reader to the articles by Clapp and Fernandez \cite{Clapp} and Fernandez and Petean \cite{Fernandez-Petean} for recent results on the multiplicity of nodal solutions of the Yamabe equation.

In order to state our main results let us introduce some notations and  definitions.

Let $f:M\longrightarrow \re$ be a function on $M$. We will denoted with $S_f$ the set of functions on $M$ that are constant along the level sets of $f$. 

\begin{Definition} Let $M$ be a closed differentiable manifold. We say that a triplet $(M,g,f)$ is an isoparametric system if $f$ is an isoparametric function on the Riemannian manifold $(M,g)$ and the scalar curvature $s_g$ belongs to $S_f$.
\end{Definition}

We will show the following theorem.

\begin{Theorem}\label{NodalYamabe} Let $(M^m,g)$ be a closed connected Riemannian manifold  and let  $(N^n,h)$ be any closed Riemannian ma\-ni\-fold of constant scalar  curvature  such that $m+n\geq 3$ and $s=s_g+s_h>0$. If $(M,g,f)$ is an isoparametric system then there exists a nodal solution of the  Yamabe equation of $(M\times N,g+h)$ that belongs to $S_f$.
 \end{Theorem}

By Theorem \ref{NodalYamabe} and the comments made above concerning  with the level sets of a proper isoparametric function we obtain the following result.

\begin{Corollary}\label{zeroset} Let $(M,g)$ and $(N,h)$ be closed connected Riemannian manifolds as is the Theorem \ref{NodalYamabe}. Let $(M,g,f)$ be an isoparame\-tric system  such that $f$ is a proper isoparametric function. Then there exists a nodal solution of the Yamabe equation of $(M\times N,g+h)$ whose nodal set \emph{(}{\em{i.e.,}} the zero sets\emph{)} is $S\times N$, where $S$ is an isoparametric hypersurface of $M$ induced  by $f$. 
\end{Corollary} 

In Section \ref{Constantscalarcurvature} (Proposition \ref{constantscalarsubcritical}) we will prove the existence of constant scalar curvature metrics in the $f-$conformal class $[g+h]_f$, {\em{i.e.,}} the set of Riemannian metrics of the form $\phi(g+h)$ with $\phi\in S_f$. Also in Section \ref{Constantscalarcurvature} we will prove some results on the nodal sets of the eigenfunctions of the conformal Laplacian.

When the dimension of each connected components of the level sets of an isoparametric function are greater than or equal to $1$, then we prove the existence of nodal solutions for the Yamabe equation.  Let $d(t)$ be the minimum of the dimensions  of $M_t$'s connected components. Let $k(f):=\min_{t\in[t_{\min},t_{max}]}\{d(t)\}$. We obtain the following result.

\begin{Theorem}\label{NodalOrbitsgeq1} Let $(M^n,g,f)$ be an isoparametric system such that $(M,g)$ is a closed connected Riemannian manifold of dimension $n\geq 3$ with non-negative scalar curvature. If $k(f)\geq 1$,  then there exists a nodal solution of the  Yamabe equation of $(M,g)$ that belongs to $S_f$.  
\end{Theorem}

With the same hypothesis of Theorem \ref{NodalOrbitsgeq1} we will prove in Section \ref{critical equation} (Coro\-llary \ref{constantscalarcritical}) the existence of constant scalar curvature Riemannian metrics in the $f-$conformal class $[g]_f$.



\section{Preliminaries}

\subsection{Examples}
In this section we  introduce a few  examples and review some constructions of families of isoparametric functions on closed Riemannian manifolds.

The simplest example of an isoparametric function on a closed Riemannian manifold is the following.
\begin{Example} Let $S^n\subset \re^{n+1}$  be the embedded $n-$dimensional sphere of radius 1, that is $S^n:=\{x\in \re^{n+1}:x_1^2+\dots +x_{n+1}^2=1\}$. Let us consider $S^n$ endowed with the round metric $g^n_0$ of curvature $1$. We define the function $f_i:S^n\longrightarrow \re$ by $f_i(x):=x_i$. These functions are isoparametric on $(S^{n},g^n_0)$. Indeed, $\|\nabla_{g_0^n}f_i\|^2_{g^n_0}=1-f_i^2$ and $\Delta_{g^n_0}f_i=-nf_i$.  The focal submanifolds of $f_i$ are $S^n_{-}\cup S^{n}_{+}=\{-e_i, e_i\}$ and the isoparametric hypersurfaces induced by $f_i$ are the $n-1$-dimensional spheres $S^n_t=\{x\in S^{n+1}:<x,e_i>=t\}$.     
\end{Example}

\begin{Example}
Let $m,n$ be integers such that $m\geq 1$, $n\geq 0$ and $m+n\geq 2$. Let consider the function $f_{m,n}$ in $S^{m+n}$  defined by $f(x):=x_1^2+\dots+x_m^2-x_{m+1}^2-\dots-x_{m+n+1}^2$. We have $\|\nabla_{g_0^{m+n}}f_{m,n}\|^2_{g^{m+n}_0}=4(1-f_{m,n}^2)$ and    $\Delta_{g^{m+n}_0}f_{m,n}=(m-n-1)-2(n+1)f_{m,n}$. Therefore, $f_{m,n}$ is an isoparametric function. Notice that when $m+n=2$ the focal submanifolds and the induced isoparametric hypersurfaces of $f_{m,n}$  are non-connected. In particular,  when $m=2$ and $n=0$, $S^2_+=S^1\times \{0\}$ and $S^2_-=\{-e_3,e_3\}$ which have different dimensions, and the isoparametric hypersurfaces induced by $f_{2,0}$ are diffeomorphic to two copies of $S^1$.  
\end{Example}

The preceding examples are particular cases  of an  important ge\-neral cons\-truction that we briefly describe below. When the manifold  has a large isometry group this construction provides  many  examples  of isoparametric functions. Note that in both examples the induced isoparame\-tric hypersurfaces are homogeneous, {\em i.e.,} there exists a subgroup of the isometry group of the ambient manifold that acts transitively on the hypersurface. 
   Let $(M,g)$ be  a closed manifold and let $G$ be a compact subgroup of the isometry group  of $(M,g)$ that acts on $M$  
with co-homogeneity one. Let  $\pi: M\longrightarrow M/G$ be the projection map. If $u$ is a smooth function on $M/G$, then $u\circ \pi$ is an isoparametric function of $(M,g)$, see for instance (Proposition 2.4, \cite{Ge-Tang}). 
Recall that not all the  isopara\-me\-tric functions induce homogeneous isoparametric hypersurfaces (see for instance \cite{Ozeki-Takeuchi-I}, \cite{Ozeki-Takeuchi-II} and \cite{Ferus}). Although for some manifolds this method provides the only possible examples.

There is another method to construct isoparametric functions when the ambient manifold is the total space  of a Riemannian submersion. Assume that $\pi:(M,g)\longrightarrow (N,h)$ is a Riemannian submersion such that all the fibers are totally geodesics.  So, if $f$ is an isoparametric function on $(N,h)$, then $f\circ \pi$ is an isoparametric function on $(M,g)$, see  (Proposition 2.5, \cite{Ge-Tang}).

\subsection{Second Yamabe constant}\label{secondyamabe}

It will play a key role in the proof of Theorem \ref{NodalYamabe} and Theorem \ref{NodalOrbitsgeq1} the conformal invariant so called {\it the second Yamabe cons\-tant}, actually a modification of it (see Section \ref{kIsoYamabeconstant}). The second Yamabe cons\-tant  $Y^2(M,g)$ was introduced by Ammann and Humbert in \cite{Ammann-Humbert} in order to study the existence of nodal solutions of the Yamabe equation. 

The conformal Laplacian $L_g$ of a closed Riemannian manifold $(M,g)$ of dimension $n\geq 3$ is the linear elliptic operator that appears in the left hand side of the Yamabe equation \eqref{Yamabeeq}, that is    
 $$L_g:=a_n\Delta_g+s_g.$$ 
It is well known that its spectrum is a non-bounded and non-decreasing sequence of eigen\-va\-lues.  We  write the spectrum of $L_g$  as the sequence   $\{\lambda_i(g)\}$  where each eigenvalue appears repeated according to its multiplicity.

The second Yamabe constant of $(M,g)$ is defined as follows.

\begin{equation}\label{2Yamabeconstant}
Y^2(M,g):=\inf_{h\in[g]}\lambda_2(h)vol(M,h)^{\frac{2}{n}},                                                                                                                  \end{equation}
where $vol(M,h)$ is the volume of $M$ with respect to the Riemannian metric $h$.

This invariant is never achieved by a Riemannian metric if the manifold is connected. However,  it could  be attained if instead of $[g]$  we consider the  {\it conformal class of generalized metrics}, {\em{i.e.}}, symmetric tensors of the form $\phi g$, where $\phi\in L^{p_n}(M)$ is a non-negative  function that is not the  zero function. Ammann and Humbert proved in \cite{Ammann-Humbert} that if  $Y^2(M,g)$ is non-negative and it is attained by a generalized metric, then there exists a nodal solution of the Yamabe equation (see Theorem 1.6, \cite{Ammann-Humbert}).

\section{$f$-conformal class and isoparametric systems}

Let $(M,g)$ be a closed Riemannian manifold and $f$ a smooth function on $M$.  
   We say that a Riemannian metric $h$ belongs to the {\it $f-$conformal class of $g$}  if $h=\phi g$ with $\phi\in S_f$. We will denote this subset of $[g]$ by $[g]_f$.

When $f$ is an isoparametric function the $f-$conformal class of  $g$ is the subset of conformally equivalent metrics  to $g$ such that  $f$ is an isoparametric function as well. More precisely, we have the following observation.

\begin{Proposition}
Let $f$ be an isoparametric function of $(M,g)$ and let $h\in[g]$. Then,  $f$ is an isoparametric function  on $(M,h)$ if and only if $[h]_f=[g]_f$.
\end{Proposition}

\begin{proof} 
Assume that $[h]_f=[g]_f$. Then $h\in[g]_f$ and   $h=\phi g$ where $\phi=\psi\circ f$ for some smooth positive function $\psi:[t_{\min},t_{\max}]\longrightarrow \re$. Using a local coordinate system  we get

\begin{equation}\label{normgradient}\|\nabla_h f\|_h^2=\sum_{1 \leq i,j\leq n}\phi^{-1}g^{ij}\frac{\partial f}{\partial x_i}\frac{\partial f}{\partial x_j}
\end{equation}

$$=\phi^{-1}\|\nabla_g f\|_g^2=\big(\psi(f)\big)^{-1}b(f)\in S_f,$$
where to obtain the last equality we used that $f$ is an isoparametric function of $(M,g)$.  

Now we compute the Laplace-Beltrami operator of $f$ with respect to the metric $h$ in a local coordinate system. We obtain,

$$\Delta_hf=-\frac{1}{\sqrt{h}}\sum_{1 \leq i,j\leq n}\frac{\partial}{\partial x_j}\Big\{ h^{ij}\sqrt{h}\frac{\partial f}{\partial x_i}\Big\}$$
$$=-\frac{1}{\phi^{\frac{n}{2}}\sqrt{g}}\sum_{1 \leq i,j\leq n}\frac{\partial}{\partial x_j}\Big\{\phi^{\frac{n-2}{2}} g^{ij}\sqrt{g}\frac{\partial f}{\partial x_i}\Big\}$$
$$=-\frac{(n-2)\psi'(f)}{2\phi^2}\sum_{1 \leq i,j\leq n}g^{ij}\frac{\partial f}{\partial x_j}\frac{\partial f}{\partial x_i}+\frac{1}{\phi}\Delta_g(f).$$
Therefore by \eqref{gradient} and \eqref{laplacian} we have 

\begin{equation}\label{laplacian2}\Delta_hf=-\frac{(n-2)\psi'(f)}{2\phi^2}\|\nabla_gf\|^2_g+\frac{1}{\phi}\Delta_g(f)\in S_f,\end{equation}
and  the desired implication follows.

Now assume   that $f$  is  an isoparametric function on $(M,h)$ as well.  By \eqref{normgradient}  we have that $\phi^{-1}b(f)\in S_f$ for some smooth function $b$, where $\phi$ is the smooth positive function that satisfies $h=\phi g$. 
Let us assume that there exist two points $P_1, P_2$ in $M$  that belong to $M_{t_0}$  ($t_0\in [t_{\min},t_{\max}]$ ) such that  $\phi(P_1)\neq \phi(P_2)$. Then necessarily $b(t_0)$ must be zero, which implies that $t_0$ is a critical value of $f$. Since the only critical values are the global minimum and maximum,  we  have either $t_0=t_{\min}$ or $t_0=t_{\max}$.  Therefore, $\phi$ is constant along the regular level sets of $f$. Assume that $P_1$ and $P_2$ belong to $M_+$, that is  $t_0=t_{\max}$.   Let $\{t_k\}$ be an increasing sequence such that   $\lim_{k\to \infty} t_k=t_{\max}$. Each regular level set of $f$ is a tube over $M_+$. So, let us consider two sequences namely $\{Q_{i,k}\}$ ($i=1,2$) such that $Q_{i,k}\in M_{t_k}$ and  $\lim_{k\to\infty }Q_{i,k}=P_i$.  Since $\phi$ is smooth and  $\phi(Q_{1,k})=\phi(Q_{2,k})$ for all $k$, then $\phi(P_1)=\phi(P_2)=t_{\max}$ which is  a contradiction. In the same way we obtain a contradiction if we assume that $P_1,P_2$ belong to $M_-$. Hence, $\phi\in S_f$ and the proposition follows.

\end{proof}

The conformal Laplacian is an elliptic second order conformally invariant operator.   It satisfies the following  invariance property. Let $h=u^{p_n-2}g$, then for  any smooth function $v$ we have 
\begin{equation}\label{invpropLg}
L_h(v)=u^{1-p_n}L_g(uv).
\end{equation}

 It is well known that the scalar curvature of a Riemannian metric  of the form $h=u^{p_n-2}g$ satisfies the equation
 \begin{equation}\label{scalarcurvatureeq}
 s_h=u^{1-p_n}\big(a_n\Delta_gu+s_gu\big)=u^{1-p_n}L_g(u).
 \end{equation} 
 If $f$ is an isopametric function, $S_f$  is an invariant subspace of the Laplace-Beltrami operator  of $(M,g)$.  Indeed,  by a straightforward computation we  have that 
$$\Delta_g (\psi\circ f)=-\psi''(f)b(f)+\psi'(f)a(f)\in S_f.$$ 
If in addition we assume that $s_g\in S_f$,  we have that $S_f$ is an invariant subspace of  $L_g$. Therefore, from Equation \eqref{scalarcurvatureeq} we obtain that the scalar curvature of any metric that belongs to $[g]_f$ is constant along the level sets of $f$. That is, if $(M,g,f)$ is an isoparametric system, then  $(M,h,f)$ is an isoparame\-tric system for any $h\in[g]_f$ as well.

\begin{Proposition}\label{propinvLaplacian} Let $f$ be an isoparametric function on $(M,g)$. $S_f$ is an invariant subspace of $L_h$ for any $h\in [g]_f$ if and only if there exists $h_0\in [g]_f$ such that $s_{h_0}\in S_f$.
\end{Proposition}

\begin{proof}
Assume that there exists  $h_0\in [g]_f$  such that  $s_{h_0}$  is constant along the level sets of $f$. Then $f$ is an isoparametric function of $(M,h_0)$ and by the comments above it follows that $S_f$ is an invariant subspace of $L_{h_0}$. Any $h\in[g]_f$ is of the form $h=u^{p_n-2}h_0$ with $u\in S_f$. Thus, from Equation \eqref{invpropLg} we get  $L_h(S_f)\subseteq S_f$. 

On the other hand, if $L_h(S_f)\subseteq S_f$,  then $s_h=L_h(1)\in S_f$.      
\end{proof}

How usual are isoparametric systems? Proposition \ref{propinvLaplacian} tells us that if there exists one isoparametric system on $M$ then there exist infinite ones. The simplest example of an isoparametric system arise when $(M,g)$ is a Riemannian manifold of constant scalar curvature admitting an isoparametric function $f$. In this situation, if $\psi$ is a smooth strictly monotone function on the range of $f$, then for any $h\in [g]_f$ $(M,h,\psi\circ f)$ is an isoparametric system.

\begin{Example}\ Let $(M,g)$ be a closed Riemannian manifold and $G$ a compact subgroup of isometries  acting on $M$ with co-homogeneity one. Let us consider the family of isoparametric functions induced by the projection map. Since the group $G$ acts transitively by isometries on each level set of the induced isoparametric function, the scalar curvature is constant along them.   
\end{Example}

In the next sections we are going to  see results about the existence of constant scalar curvature metric in the   
$f-$conformal class. Let us now discuss this in the setting of the example above.   
Let $f$ be the induced isopametric function on $(M,g)$ by a co-ho\-mo\-geneity one action of a compact subgroup of isometries $G$. Let us consider the set of $G-$invariant conformally equivalent metrics to $g$, namely $[g]_G$.
 Let us consider the Yamabe functional $Y$ restricted to the  $[g]_G$. That is,
 $$h\in[g]_G\longrightarrow Y(h):=\frac{\int_Ms_hdv_h}{vol(M,h)^{\frac{n-2}{n}}}=\frac{\int_ML_g(u)udv_g}{\|u\|_{p_n}^2}$$  
\noindent
if $h=u^{p_n-2}g$.
The infimum of the Yamabe functional restricted to $[g]_G$  is the so called $G-$equivariant Yamabe constant $Y_G(M,g)$. Note that when $G$ is the trivial group $Y_G(M,g)$ is just the Yamabe constant of $(M,g)$. The problem whether this infimum is attained or not is known as the $G-$equivariant Yamabe problem. 
If $Y_G(M,g)$  is attained then there exists a minimizing $G-$invariant metric of constant scalar curvature conformal to $g$ (see for instance \cite{Hebey-Vaugon}). This minimizing metric is of the form $h_0=\phi g$ with $\phi$ a $G-$invariant function, hence it belongs to $[g]_f$.  The $G-$equivariant Yamabe constant is attained if the following strict inequality holds 
\begin{equation}\label{HebeyVaugonineq}
 Y_G(M,[g]_G)<Y(S^n)\Lambda_{G}^{\frac{2}{n}},
\end{equation}
\noindent
where $Y(S^n)=n(n-1)vol(S^n,g^n_0)^{\frac{2}{n}}$ is the Yamabe constant of the round sphere $(S^n,g^n_0)$ and $\Lambda_{G}$ is the infimum of the cardinal of the orbits induced by the action of $G$ on $M$.  
Assuming the  positive mass theorem\footnote{Recently a proof of the positive mass theorem was announced independently by Schoen and Yau in \cite{Schoen-Yau} and  by Lohkamp in a series of papers (see \cite{Lohkamp1} and  \cite{Lohkamp2}).} by the works of Hebey and Vaugon \cite{Hebey-Vaugon} and Madani \cite{Madani} the Inequality \eqref{HebeyVaugonineq} holds, unless the action has a fixed point or the manifold is conformal to the round sphere. Therefore, the $G-$equivariant constant is attained. Note that the Inequality \eqref{HebeyVaugonineq} holds trivially if either the orbits are not finite or the equivariant Yamabe constant is non-positive (which is equivalent to the existence of a metric with non-positive scalar curvature in $[g]$).

Another example of an isoparametric system is provided by the fo\-llowing construction.

\begin{Example}
Let $(N^n,h)$ be a closed Riemannian manifold of dimension $n\geq 2$.  Let $\varphi$ be an smooth positive function on the $1-$dimensional sphere and consider the warped product $(N\times S^1, \varphi h+dt^2)$. Let $\pi_{S^1}:N\times S^1\longrightarrow S^1$ be the projection with respect to $S^1$.  Since the fibers of $\pi_{S^1}$ are totally umbilical, it is well known that there exists a family of isopametric functions on $(N\times S^1, \varphi h+dt^2)$ such that the induced isoparametric hypersurfaces  are $N\times \{t\}$. Let $f$ be one isoparametric function of this family.  Computing the  scalar curvature of the warped product, see (Theorem 2.1, \cite{Dobarro-LamiDozo}), we obtain 
$$s_{\varphi h+dt^2}=\varphi^{-\frac{n+1}{4}}\Big(\frac{4n}{n+1}\Delta_{dt^2}(\varphi^{\frac{n+1}{4}})+s_h\varphi^{\frac{n-3}{4}} \Big).$$
\noindent
Therefore, if $h$ is a metric of constant scalar curvature then $s_{\varphi h+dt^2}\in S_f$. 
\end{Example}

\begin{Remark}\label{ExnonIsosystem} The scalar curvature of a Riemmannian manifold is not necessarily constant along the level sets of an isoparametric function. For instance, let $(M,g)$ be a closed Riemannian manifold  and let $(N,h)$ be a Riemannian manifold with non-constant scalar curvature.  If $f$ is an isoparametric function on $(M,g)$ then so is it as a function on the Riemannian product $(M\times N, g+h)$. However, the level sets of $f$ are $\{M_t\times N\}$,  thus $s_{g+h}\notin S_f$.

\end{Remark}

\section{Constant scalar curvature metrics on $f-$conformal classes}\label{Constantscalarcurvature}

In this section we prove the existence of constant scalar curvature metrics in some $f-$conformal classes of a Riemannian product. When the Riemannian manifold does not admit a conformal metric of positive scalar curvature,  there exists essentially only one metric  of constant scalar curvature (up to scalar dilations) in the conformal class. Hence, throughout this section and in the rest of the paper we will  address the positive scalar curvature  setting.

\subsection{Subcritical Yamabe equation and the spectrum of $L_g$ on the $f-$conformal class.}
Let $(M,g)$ be a closed Riemannian manifold of dimension $n\geq 3$ and let $(M,g,f)$ be an isoparametric system. For $2\leq s \leq p_n$, let us consider the functional given by
$$v\longrightarrow J^s(v):=\frac{\int_ML_g(v)vdv_g}{\|v\|_s^2}. $$ Let $$\alpha_s^f(M,g)=\inf_{A-\{0\}}J^s(v)$$ where $A:=H^2_1(M)\cap S_f$.

\begin{Lemma}\label{Subcritical}  Let $2\leq s<p_n$. There exists a positive  smooth solution $u$ that belongs to $S_f$ of the subcritical Yamabe equation
 \begin{equation}\label{subcriticaleq}
              L_g(u)=\alpha_s^f(M,g)u^{s-1}.
             \end{equation}
\end{Lemma}
 The proof of Lemma \ref{Subcritical} is essentially the same to the one used to prove   existence of solutions of the subcritical Yamabe equation, see for instance (Proposition 2.2,  \cite{A-F-P}) and (Lemma 5.1, \cite{Henry-Petean}).
 \begin{proof} Let  $\{v_i\}$ be  a non-negative minimizing sequence of $J^s$ restricted to $A$. We can assume that each $v_i$ is a smooth function and $\|v_i\|_{s}=1$. It is easy to see that $v_i$ is a bounded sequence in $H^2_1(M)$. Hence,  there exists $v\in H^2_1(M)$ such that (up to a subsequence) $v_i$ converges weakly to $v$ in $H^2_1(M)$. By the Rellich-Kondrakov theorem ($s<p_n$) we have that the inclusion map $i:H^2_1(M)\longrightarrow L^{s}(M)$  is a compact operator. So, again up to a subsequence, $\{v_i\}$ converges strongly in $L^{s}(M)$ to $v$, which implies that $\{v_i\}$ converges to $v$ a.e.  and $\|v\|_s=1$.  Since $S_f$ is a close subspace of $H^2_1(M)$, $v$ is a non-negative minimizer of $J^s$. Therefore, we have that $$\int_M\Big(L_g(v)-\alpha_s^f(M,g)v^{s-1}\Big)\psi dv_g=0$$ for any $\psi\in A$.
   Since $L_g(v)-\alpha_s^f(M,g)v^{s-1}\in S_f$,  necessarily $L_g(v)=\alpha_s^f(M,g)v^{s-1}$. Then the lemma follows using standard re\-gu\-larity arguments and the strong maximum principle. 
     \end{proof}

\begin{Remark} The statement of Lemma \ref{Subcritical} is still valid if instead of the scalar curvature we place a smooth non-negative function $\varphi\in S_f$ in Equation \eqref{subcriticaleq}.
 \end{Remark}

Under the same hypothesis of Lemma \ref{Subcritical} we have the following proposition.  
 
\begin{Proposition}\label{fspectrum} There exists an unbounded  subsequence  of the spectrum of $L_g$ such that the associated eigenfunctions belong to $S_f$.   
\end{Proposition}

 \begin{proof} Applying Lemma \ref{Subcritical} for $s=2$ we obtain a smooth positive function $w_1\in S_f$ that satisfies the equation $L_g(w_1)=\lambda^f_1(M,g)w_1$ where $\lambda^f_1(M,g)=\alpha_2^f(M,g)$. Therefore, $w_1$ is an eigenfunction associated to the eigenvalue   $\lambda^f_1(M,g)$. Note that $\lambda^f_1(M,g)$ is the smallest eigenvalue of $L_g$ that has an associated eigenfunction in $S_f$. Let $A_2$ be the closed subspace of $H^2_1(M)$ defined by $A_2:=A_1\cap \{u: \int_Mu w_1dv_g=0\}$.  We define $\lambda_2^f(g):=\inf_{w\in A_2-\{0\}}J^2(w)$. Using the same argument as the one used in the proof of Lemma \ref{Subcritical} we obtain $w_2\in S_f$ that achieves   $\lambda_2^f(g)$ and satisfies $L_g(w_2)= \lambda_2^f(g)w_2$. Using the strong maxi\-mum principle it can be shown that $\lambda_1^f(g)<\lambda_2^f(g)$. Then we define $A_3:=A_2\cap \{u: \int_Mu w_2dv_g=0\}$ and we repeat the procedure to obtain $\lambda_3^f(g)$ and $w_3$, and so on.      
   \end{proof}

 We will note the sequence obtained in the proof of Proposition \ref{fspectrum} by
\begin{equation}\label{f-spectrum}
 \lambda_1^f(g)<\lambda_2^f(g)\leq \lambda_3^f(g)\leq \dots \leq \lambda_k^f(g)\nearrow +\infty\subseteq{Spec(L_g)},
 \end{equation}
\noindent
where each eigenvalue appears repeated according to its multiplicity.   
Actually  this sequence is the spectrum of $L_g$ restricted to the subspace $S_f$. Note that since we are considering metrics that are conformally equivalent to a positive scalar curvature metric, the first eigenvalue of the conformal Laplacian is positive, hence  $\lambda_1^f(g)>0$.

 An immediate consequence of Proposition \ref{fspectrum} is the following coro\-llary.

 \begin{Corollary} Let $(M,g,f)$ be an isoparametric  system.  Then there exist infinitely many eigenfunctions of $L_g$ whose nodal sets are a union of level sets of $f$.
  \end{Corollary}

\begin{Remark} Let $(M,g,f)$ be an isoparametric system where $(M,g)$ is  a co\-nnected closed Riemannian manifold of non-negative scalar curvature and $f$ is a proper isoparametric function. For any metric $h\in [g]_f$, the eigenfunction $u_2^h$  associated to the eigenvalue $\lambda_2^f(h)$ changes sign. Actually,  the connectedness of $M$ implies that the nodal set of  $u_2^h$ is non-empty and the number of connected components of $M-\{(u_2^h)^{-1}(0)\}$ is two.  By the comments  made in the Introduction, we know that all the isoparametric hypersurfaces induced by $f$ are connected and separates $M$ into two disc bundles over the focal submanifolds.  Therefore, the nodal set of $u_2^h$ consist in one of the isoparametric hypersurfaces induced by $f$ and it eventually include one of the components of the focal  submanifold $\{M_{-}\cup M_{+}\}$.  However, using the strong maximum principle we obtain that neither $M_+$ nor $M_-$ could belong to the zero set of $u_2^h$ \emph{(}see the proof of Corollary \ref{zeroset}\emph{)}.  
Therefore, in this situation we conclude that there exist infinitely many Riemannian metrics on $M$ such  that they have an eigenfunction of the conformal Laplacian which its nodal set is an isoparametric hypersurface induced by $f$. 
 \end{Remark}

  \subsection{Isoparametric $k^{th}$ $f-$Yamabe constant.}\label{kIsoYamabeconstant}
  Now, we are in condition to introduced the  $isopametric\ k^{th}\ f-Yamabe\ constant$ of an isoparametric system $(M,g,f)$. It is an invariant of the $f-$conformal class that we are going to use  in Section \ref{ProofThm} to prove Theorem \ref{NodalYamabe}. The isoparametric $k^{th}-$ Yamabe constant of $(M,g,f)$ is given by 
  
  $$Y^k_f(M,[g]_f):=\inf_{h\in [g]_f}\lambda_k^f(h)vol(M,h)^{\frac{2}{n}}.$$
  
  Let $u\in L^{p_n}_{\geq 0}(M)-\{0\}\cap S_f$ and consider the generalized metric $g_u=u^{p_n-2}g$.
    We define  the $k-$Grassmanian space of $H^2_1(M)\cap S_f$ with respect to $u$, namely $Gr^k_u(H^2_{1}(M)\cap S_f)$, as the space of $k-$dimensional subspaces of $H^2_{1}(M)\cap S_f$ which are  $k$-dimensional subspaces in $H^2_{1}(M-\{u^{-1}(0)\})\cap S_f$ as well. 
    
    For $v\in  Gr^1_u(H^2_{1}(M)\cap S_f)-\{0\}$ let us consider 
     $$F(u,v):=\frac{\int_M a_{n}|\nabla v|^2_g+s_gv^2dv_{g}}{\int_M u^{p_n-2}v^2dv_{g}}.$$
    \noindent    
    Then, using the variational characterization of the eigenvalues (see for instance Proposition 2.1 in \cite{Ammann-Humbert}) it is not  difficult to see that 
  $$Y^k_f(M,[g]_f)=\inf_{\substack{ u\in L^{p_n}_{\geq 0}(M)\cap S_f -\{0\}\\ V\in Gr^k_u(H^2_{1}(M)\cap S_f)}}  \sup_{v\in V-\{0\}}F(u,v)\big(\int_Mu^{p_n}dv_g\big)^{\frac{2}{n}},$$ 
  \noindent
  where  
     $$\lambda_k^f(g_u):=\inf_{V\in Gr^k_u(H^2_{1}(M)\cap S_f)}  \sup_{v\in V-\{0\}}F(u,v)$$
is the so called   $k^{th}-$generalized eigenvalue of $L_{g_u}$ restricted to $S_f$.

 \subsection{Constant scalar curvature metrics.}
 
  Let $(M^m,g)$ and $(N^n,h)$ be closed Riemannian manifolds.  Assume that $(N,h)$ has constant scalar curvature. If $f$ is an isoparametric function on $(M,g)$ then so is  $f$ on  $(M\times N, g+h)$. Morever,  if $(M,g,f)$ is an isoparametric system then $(M\times N, g+h, f)$ is an isoparametric system as well.

\begin{Proposition}\label{constantscalarsubcritical} Let $(M^m,g)$ be a closed Riemannian manifold  and let  $(N^n,h)$ be any closed Riemannian manifold of constant scalar curvature. Assume that $m+n\geq 3$ and $s_g+s_h>0$. Let $(M,g,f)$ be an isoparametric system. Then, there exists a Riemannian metric in  $[g+h]_f$ with constant scalar curvature.  
\end{Proposition} 

\begin{proof}  Without loss of generality  we can assume that $vol(N,h)=1$. Let us  consider the Yamabe functional  of the Riemannian product $(M\times N,g+h)$ restricted to $A=H^2_1(M\times N)\cap S_f=H^2_1(M)\cap S_f$.  That is, if $w\in A$ we have  
$$J(w)=\frac{\int_{M}a_{m+n}|\nabla_g w|^2+s_{g+h}w^2dv_g}{\|w\|_{p_{m+n}}^2},$$
\noindent
where the  $L^{p_{m+n}}$ norm  is considered with respect to the metric $g$. Note that $p_{m+n}<p_{m}$, so we are in the subcritical setting as in Lemma \ref{Subcritical}. Hence, we take  any smooth non-negative minimizing sequence  $\{u_i\}$ such that $\|u_i\|_{p_{m+n}}=1$ and we proceed as in the proof of Lemma \ref{Subcritical}. We obtain a  smooth positive function $u$ that belongs to $S_f$, that depends only on $M$ and satisfies the equation $$a_{m+n}\Delta_gu+s_{g+h}u=\Big(\inf_{w\in A-\{0\}}Y(w)\Big)u^{p_{m+n}-1}.$$  
Therefore, by Equation \eqref{scalarcurvatureeq} we have that the Riemannian metric $u^{p_{m+n}-2}$ $(g+h)$ belongs to $[g]_f$ and has constant scalar curvature.
\end{proof}

When a conformal class admits a metric of positive scalar curvature, then it might have within several non-isometric metrics of constant scalar curvature. Therefore, in the positive scalar curvature setting a $f-$conformal class could has many metrics of constant scalar curvature as the following examples show.        
   
 \begin{Example}\label{csc}  Let $(M,g)$ and $(N,h)$ be both closed Riemannian ma\-nifolds of constant scalar curvature and $f$ an isoparame\-tric function on $(M,g)$.
 Applying \emph{(}Theorem 1.4, \cite{Henry-Petean}\emph{)} we obtain that if  $s_{g+h}>\frac{\lambda_2^f(g)}{p_{m+n}-1}$, then there exists a non constant function $u\in S_f$ such that the Riemannian metric $g_u=u^{p_{m+n}-2}(g+h)\in [g+h]_f$ has constant scalar curvature.      
 \end{Example}

 \begin{Example} Let us consider the Riemannian product of spheres  $(S^n\times S^m, g^n_0+tg^m_0)$ and let $f$ be an isoparametric function on $(S^n,g^n_0)$. Let $k$ be any positive integer. Then, there exists $t_0$ such that  there are at least $k$ non-isometrically unit volume metrics of constant scalar curvature on $[g^n_0+tg^m_0]_f$ if $t\leq t_0$. We refer the reader to \emph{(}Corollary 1.3, \cite{Henry-Petean}\emph{)} for a precise statement on the lower bound of the number of non-isometrically unit volume metrics with constant scalar curvature in the $f-$conformal class. Recall that for even dimensional spheres $S^{2s}$ there exist isoparametric hypersurfaces with one principal curvature of multiplicity $2s-1$ (hyperspheres), with two distinct principal curvatures (embeddings of $S^k\times S^{s-k-1}$), and for $S^4$ there  exist isoparametric hypersurfaces with three distinct principal curvatures as well.
 	For instance,  if $n=2s$ with $s\geq 1$, let $l(t)=\big(m(m-1)t^{-1}+2s(2s-1)\big)/(2s+m-1)$. Assume that $l(t)$  belongs to the interval given by $\big(A_{i}, A_{i+1}\big]$ where  $A_i=i(2s+i-1)$. Then,  we have that the number of essentially different constant scalar curvature metrics in $[g^{2s}_0+tg^m_0]_f$ is at least $i+[(2s-1)/2][i/2]$ if $s>2$ and $i+[i/2]+[i/3]$ if $s=2$.

 \end{Example}

\section{Proof  of Theorem \ref{NodalYamabe}.}\label{ProofThm}

We are going to prove that the  isoparametric second $f-$Yamabe constant $Y^2_f(M\times N,[g+h]_f)$ is attained  by a generalized metric  $g_u=u^{p_{m+n}-2}(g+h)$. Actually, $u$ will be the absolute value of a nodal solution of the Yamabe equation. The proof of Theorem \ref{NodalYamabe} is essentially the proof of (Theorem 1.1, \cite{Petean2}) which states that the $M-${\it second Yamabe constant} of $(M\times N, g+h)$, which is a subcritical version of the second Yamabe constant, is always attained. The achievement of the $M-$second Yamabe constant leads to the existence of a nodal solutions of the Yamabe equation that depend only on $M$.

\begin{proof}[Proof of Theorem \ref{NodalYamabe}.]

Without loss of generality we can assume that $vol(N,h)=1$.

For any  $u\in L^{p_{m+n}}(M)\cap S_f-\{0\}$, there exist $v_1\geq 0$ and $v_2$ that belong to $H^2_{1}(M)\cap S_f$ such that 

\begin{equation}\label{orth}
\int_Mu^{p_{m+n}-2}v_iv_jdv_g=\delta_{ij},
\end{equation} 
\noindent
and at the same time they satisfy in a weak sense the following equations 
\begin{gather}
a_{m+n}\Delta_g v_1+s_{g+h}v_1=\lambda_{1}^f(g_u)u^{p_{m+n}-2}v_1,\label{v_1}
\\
a_{m+n}\Delta_g v_2+s_{g+h}v_2=\lambda_{2}^f(g_u)u^{p_{m+n}-2}v_2.\label{v_2}
\end{gather} 

For the details we refer to the reader to  (Proposition 3.2 in  \cite{Ammann-Humbert}) and to Section 4 in \cite{Henry-Madani}. The subspace $V_0:=span(v_1,v_2)$ realizes the gene\-ralized second eigenvalue $\lambda_2^f(g_u)$. More precisely, we have 
$$\lambda_2^f(g_u)=\sup_{w\in V_0-\{0\}}F(u,w)=F(u,v_2).$$

Using similar arguments as the ones used in the proof of (Theorem 3.4, \cite{Ammann-Humbert}) and (Theorem 1.5,  \cite{Henry-Madani}) it can be seen  that if  $g_u$ realizes the second isoparametric $f-$constant of $(M\times N,g+h)$ then  $u=|v_2|$ and $v_2$ is a nodal solution of the Yamabe equation. Also, by the connectedness of $M$, it can be seen that the number of  connected components of $M\times N-\{v_2^{-1}(0)\}$ is two.

In our setting there always exists a minimizer of $Y_f^2(M\times N,[g+h]_f)$ thanks to $p_{m+n}$ is less than the critical Sobolev exponent $p_m$. Indeed, let $\{u_i\}$ be a minimizing sequence of positive smooth functions that belong to $S_f$  and satisfy $\|u_i\|_{p_{m+n}}=1$.

For $u_i$, let $v_1^i$ and  $v_2^i$ be the functions that satisfy (\ref{orth}), (\ref{v_1}) and (\ref{v_2}).  Then, $$\lambda^f_2(g_{u_i})=F(u_i,v_2^i),$$ 
$$\lambda^f_2(g_{u_i})\longrightarrow_{i\to +\infty} Y^2_f(M\times N,[g+h]_f)$$
and 
\begin{equation}\label{limitelambda1}
\limsup_{i\to +\infty} \lambda_1^f(g_{u_i})\leq Y_f^2(M\times N,[g+h]_f).
\end{equation}
The sequences $\{u_i\}$, $\{v^i_1\}$ and $\{v^i_2\}$ are bounded in $L^{p_{m+n}}(M)$ and $H^2_1(M)$, respectively. Therefore, up to a subsequence, there exist $u$, $v_1$ and $v_2$ such that $u_i$ converges weakly to $u$ in $L^{p_{m+n}}(M)$, and $v_i^j$ converges weakly to $v_j$  in $H^2_1(M)$ (for $j=1,2$). As a consequence of the weak convergence we have that 
\begin{equation}\label{liminfnorm}
\|u\|_{p_{m+n}}\leq \liminf_{i\to +\infty} \|u_i\|_{p_{m+n}}=1.
\end{equation}
By the Rellich-Kondrakov theorem  we have that $\{v^i_j\}$ converges strongly to $v_j$ in $L^{p_{m+n}}(M)$  (again up to a subsequence) and $v_j\in S_f$ (for $j=1,2$). This implies that the following equations hold in a weak sense 

$$a_{m+n}\Delta_g v_1+s_{g+h}v_1=\Big(\limsup_{i\to \infty} \lambda_1^f(g_{u_i})\Big)u^{p_{m+n}-2}v_1,$$
and
$$a_{m+n}\Delta_g v_2+s_{g+h}v_2=Y_f^2(M\times N,[g+h]_f)u^{p_{m+n}-2}v_2.$$

Let $W_0:=span (v_1,v_2)\in Gr^2_u(H^2_{1}(M)\cap S_f)$. By definition of the isoparametric second $f-$Yamabe constant we have 
$$Y_f^2(M\times N,[g+h]_f)\leq sup_{v\in W_0-\{0\}}F(u,v)\big(\int_Mu^{p_{m+n}}dv_g\big)^{\frac{2}{m+n}}.$$
Then  by inequalities \eqref{limitelambda1} and \eqref{liminfnorm} we obtain  
$$sup_{v\in W_0-\{0\}}F(u,v)\big(\int_Mu^{p_{m+n}}dv_g\big)^{\frac{2}{m+n}}\leq Y_f^2(M\times N,[g+h]_f).$$
Hence, $g_u$ realizes the isoparametric second $f-$Yamabe constant and the theo\-rem follows.
\end{proof}

\begin{proof}[Proof of Corollary \ref{zeroset}] Since $f$ is a proper isoparametric function all the level sets are connected submanifolds of $M$. If $S$ is an isoparametric hypersurfaces induced by $f$, $M-S$ is a disjoint union of two  connected open sets  $\Omega_-$ and $\Omega_+$,  where $M_-\subset \Omega_-$,  $M_+\subset \Omega_+$ and $\partial (\overline{\Omega_-})=\partial (\overline{\Omega_+})=S$. Actually, $\Omega_{+}$ ($\Omega_-$) is a disk bundle over the focal submanifold $M_{+}$ ($M_-$). Let $w\in S_f$ be the nodal solution of the Yamabe equation  provided by Theorem \ref{NodalYamabe}.  The nodal set $Z$ of $w$  is a union of finite level sets of $f$. But  $M-Z$ has only two connected components. Then, it is a subset of $S\cup M_-\cup M_+$ where $S$ is an isoparametric hypersurface induced by $f$.  The set where $w$ is positive $\{w>0\}$ is included either in $\Omega_{+}$ or $\Omega_-$. Let us assume that $\{w>0\}\subseteq \Omega_+$ (we will arrive to the same conclusion if we assume that $\{w>0\}\subseteq \Omega_-$) and  is a proper subset. Then $M_+\subset Z$. However, using the strong maximum principle (recall that the scalar curvature is positive), we conclude that $w$ can not be entirely  zero in $M_+$ unless $w$ is the zero function, which is not possible. 
Therefore $\{w>0\}=\Omega_+$. In the same way we see that $M_-$ does not belong to the zero set of $w$. Therefore, $\{w<0\}=\Omega_{-}$. Finally, considering $w$ as a function of $M\times N$ we obtain that its nodal set is $S\times N.$ 
 
\end{proof}

\section{Critical equation}\label{critical equation}

In this section we will prove Theorem \ref{NodalOrbitsgeq1}. This theorem is a  consequence of the compact embedding of $H^2_1(M)\cap S_f$ into $L^{p_n}(M)$ when the dimensions of the level sets of $f$ are  positive. In order to prove the compactness of this embedding we will prove Lemma \ref{fSobolevembedding} below, which says that when each dimension of all connected components of any level set of an isoparametric function $f$ is  positive then the Rellich-Kondrakov theorem can be improved by considering the functions that are constant along the level sets of $f$.
When the isoparametric function is induced by a co-homegeneity one action of a compact subgroup of the isometry group, Lemma \ref{fSobolevembedding} is the Sobolev embedding theorem under the presence of symmetries proved in (Corollary 1, \cite{Hebey-Vaugon3}).  


Given an isoparametric function on $(M,g)$ we denote by $H^q_ {1,f}(M)$ the closed subspace $H^q_1(M)\cap S_f$.

\begin{Lemma}\label{fSobolevembedding} Let $(M,g)$ be a closed Riemannian manifold of dimension $n$ and let $f$ be an isoparametric function on $(M,g)$ such that $k(f)\geq 1$. We have that
\begin{enumerate}
 \item[i)] if $q\geq n-k(f)$, then the inclusion map of $H^q_{1,f}(M)$  into $L^p(M)$ is continuous and compact for any $p\geq 1$.
 \item[ii)] If $q< n-k(f)$, then the inclusion map of $H^q_{1,f}(M)$  into $L^p(M)$ is continuous for any $p$ that satisfies $1\leq p \leq \frac{q(n-k(f))}{n-k(f)-q}$, and compact if $1\leq p < \frac{q(n-k(f))}{n-k(f)-q}$.   
  \end{enumerate}

\end{Lemma}

\begin{proof}The  proof of the lemma is divided in three steps. 

{\bf Step 1:}  In this step using the structure of the isoparametric foliations we are going to construct an appropriate coordinate systems that will allow us to decrease the dimension in order to use a subcritical Sobolev embedding instead the critical one. The existence of these coordinates systems are consequence of the structural theorem of isoparametric hypersurfaces mentioned in the Introduction (Theorem A, \cite{Wang}). Recall that it states that any regular level set $M_t$ ($t_{\min}<t<t_{\max}$) is a tube over either of the focal submanifolds. On the other hand, for any $x$ that belongs to a regular level set $M_t$, the integral curve $c_x(t)$ of the vector field $F=\nabla f/\sqrt{b(f)}$ trough $x$ is a unit speed geodesic that intersects orthogonally all the level set of $f$. The length of the segments induced by the integral curves of $F$ realises the distance between to regular level set. We have  $$d_g(M_{f(c_x(t_1))},M_{f(c_x(t_0))})=t_1-t_0.$$    
So, let $x\in M$. Let $k(x)$ be the dimension of the connected component of $M_{f(x)}$ where $x$ belongs to. We denote with $\pi_2$ the projection of $\re^k\times \re^{n-k}$ onto $\re^{n-k}$.  Using Fermi coordinates centered at $x$ on $M_{f(x)}$, it is not difficult to see that there exists a coordinate system $(W,\varphi)$  (shrinking suitably the neighborhood $W$) such that:

\begin{itemize}
\item[1)] $x\in W$ and $\varphi(W)=U\times V$ where, $U\subset \re^{k(x)}$ and $V\subset \re^{n-k(x)}$ are open subsets with smooth boundaries. 
\item[2)] For any $y\in W$, $U\times \pi_2(\varphi(y))\subseteq \varphi( M_{f(y)}\cap W)$. 
\item[3)] There exists a positive constant $a_x$ such that $(a_x)^{-1}dv_{g^n_e}\leq dv_g\leq a_x dv_{g^n_e}$,
where $dv_{g^n_e}$ is the volume element with respect to the Euclidean metric.
\end{itemize}

{\bf Step 2}  In this step we will prove that the inclusion $H^q_{1,f}(M)\subset L^p(M)$ is a continuous map for any $p\geq 1$ if either $q\geq n-k(f)$ or $q< n-k(f)$ and $ p\leq q(n-k(f))/(n-k(f)-q)$. 

Since $M$ is a compact manifold there exists a collection of points $\{x_i\}_{i=1}^m$ such that the charts  $(W_i,\varphi_i)_{i\in [1,m]}$  (where each $x_i\in (W_i,\varphi_i)$ is one of the coordinates system introduced in the first step) cover $M$.   Let $k_i=k(x_i)$. 

Let $u\in C^{\infty}(M)\cap S_f$ and $p\geq 1$. Let us consider $u_i:V_i\longrightarrow \re$ the smooth function defined by $u_i(y):=u(\varphi^{-1}_i(x,y))$ for any $x\in U_i$. The function $u_i$ is well defined by property $2)$.

By the properties of the coordinate systems $(W_i,\varphi_i)$ there exist $A_i, B_i,$ $ C_i$ positive constants such that  $B_i \|u_i\|_{L^s(V_i)}\leq \|u\|_{L^s(W_i)}\leq C_i \|u_i\|_{L^s(V_i)}$ and $D_i\|\nabla u_i\|_{L^s(V_i)}\leq \|\nabla u\|_{L^s(W_i)}$ for any $s\geq 1$.

By the above inequalities and applying the  Sobolev embedding theorem for Euclidean bounded domains to function $u_i$  we obtain that 
\begin{equation}\label{ineqSobolev1}\|u\|_{L^p(W_i)}\leq A_i\Big(\|\nabla u\|_{L^q(W_i)}+\|u\|_{L^q(W_i)} \Big)
\end{equation} 
if either $n-k_i\leq q$ or $n-k_i> q$ and  $p\leq q(n-k_i)/(n-k_i-q)$.

Now if $n-k(f)\leq q$, then $n-k_i\leq q$ for any $1\leq i\leq m$. If $n-k(f)>q$ and $p\leq q(n-k(f))/(n-k(f)-q)$ , then for each $i$ we have that either $n-k_i\leq q$ or $n-k_i> q$ and $p\leq (n-k_i)q/(n-k_i-q)$. Indeed, since $n-k\geq n-k_i$ we have that $(n-k)(n-k_i-q)\leq (n-k_i)(n-k-q)$ which implies that $(n-k)q/(n-k-q)\leq (n-k_i)q/(n-k_i-q)$. Therefore by Inequality (\ref{ineqSobolev1}) we have    
  
$$\|u\|_{p}\leq \sum_{i=1}^m\|u\|_{L^p(W_i)}\leq A\sum_{i=1}^{m}\Big(\|\nabla u\|_{L^q(W_i)}+\| u\|_{L^q(W_i)} \Big)$$  
  $$\leq B\Big( \|\nabla u\|_{L^q(M)}+\| u\|_{L^q(M)} \Big).$$

{\bf Step 3} Let $\{u_j\}$ be a bounded sequence in $H^q_{1,f}(M)$.  Let $\{\phi_i\}_{i=[1,m]}$ be a partition of  unity subordinate to the covering $\{(W_i,\varphi_i)\}_{i\in [1,m]}$ such that $\phi_i(\varphi_i^{-1}(x,y))=\phi_i(\varphi_i^{-1}(x',y))$ for any $1\leq i\leq m$,  $x,x'\in U_i$ and  $y\in V_i$. We define the sequence  $u_{j,i}:=(\phi_i\circ \varphi^{-1}) (u_j)_i$. Notice that  functions $u_{j,i}$ has compact support in $V_i$. For any $1\leq i\leq m$, $\{u_{j,i}\}$ is a bounded sequence in $H^q_1(V_i)$. If either $q\geq n-k(f)$ or $q< n-k(f)$ and  $1\leq p < \frac{q(n-k(f))}{n-k(f)-q}$  then we have  
\begin{equation}\label{Ineqexponentes2}
\frac{1}{p}>\frac{1}{q}-\frac{1}{n-k_i}                                                                                                                                                                                                                                                                                                                                                                                                                                                                                                                                                                                                                                                                                                                                                                                                                                                                                                                                                                                                                                                                                                           \end{equation}
                                                                                                                                                                                                                                                                                                                                                                                                                                                                                                                                                    for any $i\in[1, m]$. On the other hand, we know that for any bounded open set $V\subset \re^l$,   $H^s_{1,0}(V)$ is compactly embedded in $L^{r}(V)$ if $1/r>1/s-1/l$ (see for instance Lemma 2.5, \cite{Hebey}). Therefore, by Inequality \eqref{Ineqexponentes2}  there is a subsequence of  $\{u_j\}$ (that in order to simplify the notation we denote  with $\{u_j\}$ as well) such that ${(u_{j})}_{i}$ is a Cauchy sequence in $L^p(V_i)$ for any $1\leq i\leq m$.   
Joining  with property $3)$  satisfied by the coordinate systems $(W_i,\varphi_i)$  we see  that the sequence $\phi_iu_j$ is a Cauchy sequence in $L^p(M)$. Since $\{\phi_i\}_{i=[1,m]}$ is a partition of  unity we obtain that $\{u_j\}$   converges strongly in $L^p(M)$.

\end{proof}

\begin{Theorem}\label{eqcritical} Let $(M,g,f)$ be an isoparametric system such that $k(f)\geq 1$. Let $s\geq 1$.   If $k(f)=n-1, n-2$ or  $s< 2(n-k(f))/(n-k(f)-2)$, then for some constant $c$ there exists a positive solution $u$ of the equation
 $$L_g(u)=cu^{s-1}$$
 that belongs to $S_f$. 
 \end{Theorem}

\begin{proof} By Lemma \ref{fSobolevembedding} the inclusion map of $H^2_{1,f}(M)$ into $L^{s}(M)$ is compact. Then we proceed as in the proof of  Lemma  \ref{Subcritical}.
 \end{proof}
\begin{Remark}
Note that $p_n<2(n-k(f))/(n-k(f)-2)$ if $k(f)\geq 1$,  therefore  by Theorem \ref{eqcritical} we have solutions of the Yamabe equation with supercritical exponents.
\end{Remark}

We have the following corollary.

\begin{Corollary}\label{constantscalarcritical}  Let $(M,g,f)$ be an isoparametric system such that $k(f)\geq 1$. Then there exists a Riemannian metric in $[g]_f$ of constant scalar curvature. 
\end{Corollary}

\begin{proof}[Proof of Theorem \ref{NodalOrbitsgeq1}] The embedding $H^2_{1,f}(M)\subset L^{p_n}(M)$ is com\-pact, then the proof is essentially the same one of Theorem \ref{NodalYamabe}.

\end{proof}




\end{document}